\documentclass{amsart}
\usepackage{amssymb,latexsym}
\usepackage{amscd,amsthm}

\usepackage[all]{xy}
\usepackage{tikz}

\usepackage{tikz-cd}

\newtheorem{theorem}{Theorem}[section]

\newtheorem{lemma}[theorem]{Lemma}
\newtheorem{proposition}[theorem]{Proposition}
\newtheorem{corollary}[theorem]{Corollary}

\theoremstyle{definition}
\newtheorem{definition}[theorem]{Definition}

\newtheorem{remark}[theorem]{Remark}

\DeclareMathOperator{\Ext}{Ext}
\DeclareMathOperator{\Hom}{Hom}

\DeclareMathOperator{\cok}{cok}
\DeclareMathOperator{\im}{Im}

\newcommand{\cat}[1]{\mathcal{#1}}           

\newcommand{\tensor}{\otimes}

\newcommand{\class}[1]{\mathcal{#1}}   

\newcommand{\Z}{\mathbb{Z}}

\newcommand{\ch}{\textnormal{Ch}(R)}
\newcommand{\chop}{\textnormal{Ch}(R^\circ)}
\newcommand{\cha}[1]{\textnormal{Ch}(\mathcal{#1})}

\newcommand{\dwclass}[1]{dw\widetilde{\class{#1}}}
\newcommand{\exclass}[1]{ex\widetilde{\class{#1}}}

\newcommand{\rightperp}[1]{#1^{\perp}}
\newcommand{\leftperp}[1]{{}^\perp #1}

\newcommand{\homcomplex}{\mathit{Hom}}

\begin{document}

\title[Acyclic complexes of pure-projective modules]{The homotopy category of acyclic complexes of pure-projective modules}

\author{James Gillespie}
\address{J.G. \ Ramapo College of New Jersey \\
         School of Theoretical and Applied Science \\
         505 Ramapo Valley Road \\
         Mahwah, NJ 07430\\ U.S.A.}
\email[Jim Gillespie]{jgillesp@ramapo.edu}
\urladdr{http://pages.ramapo.edu/~jgillesp/}

\date{\today}

\keywords{homotopy category of chain complexes; pure projective module; derived category, abelian model category} 

\thanks{2020 Mathematics Subject Classification. 18N40, 18G35, 18G25}

\begin{abstract}
Let $R$ be any ring with identity. We show that the homotopy category of all acyclic chain complexes of pure-projective $R$-modules is a compactly generated triangulated category. We do this by constructing abelian model structures that put this homotopy category into a recollement with two other compactly generated triangulated categories: The usual derived category of $R$ and the pure derived category of $R$. 
\end{abstract}

\maketitle

\section{Introduction}

Let $R$ be a ring with identity and $K(R)$ the homotopy category of all chain complexes of (say left) $R$-modules. 
In~\cite{gillespie-K-flat} we studied the Verdier quotient $$\class{D}_{\textnormal{K-flat}}(R):=K(R)/K\textnormal{-Flat}$$ introduced by Emmanouil in~\cite{emmanouil-K-flatness-and-orthogonality-in-homotopy-cats}, where $K\textnormal{-Flat}$ denotes the thick subcategory of $K(R)$ consisting of all K-flat complexes. Recall that Spaltenstein's K-flat complexes  are those chain complexes $X$ for which the chain complex tensor product $- \otimes_R X$ preserves acyclicity~\cite{spaltenstein}. Emmanouil shows that $\class{D}_{\textnormal{K-flat}}(R)$ is equivalent to $K_{ac}(\class{PI})$, the chain homotopy category of all acyclic (not necessarily pure acyclic) complexes of pure-injective $R$-modules. Here we are letting $\class{PI}$ denote the class of all pure-injective modules, that is, the modules that are injective with respect to pure monomorphisms. 

Using the approach of Hovey's abelian model categories, it was shown in~\cite{gillespie-K-flat} that $\class{D}_{\textnormal{K-flat}}(R)$ is a compactly generated triangulated category, and that we have a recollement
\[
\xy
(-28,0)*+{\class{D}_{\textnormal{K-flat}}(R)};
(0,0)*+{\class{D}_{pur}(R)};
(25,0)*+{\class{D}(R)};
{(-19,0) \ar (-10,0)};
{(-10,0) \ar@<0.5em> (-19,0)};
{(-10,0) \ar@<-0.5em> (-19,0)};
{(10,0) \ar (19,0)};
{(19,0) \ar@<0.5em> (10,0)};
{(19,0) \ar@<-0.5em> (10,0)};
\endxy
\] 
where $\class{D}(R)$ is the usual derived category of $R$, and $\class{D}_{pur}(R)$ is the pure derived category of $R$.
Each of the three triangulated categories is the homotopy category of an abelian model structure, and in each case the fibrant objects form some class of complexes of pure-injective modules. When restricting the first two categories to their fibrant objects the recollement becomes 
\[
\xy
(-28,0)*+{K_{ac}(\class{PI})};
(0,0)*+{K(\class{PI})};
(25,0)*+{\class{D}(R).};
{(-19,0) \ar (-10,0)};
{(-10,0) \ar@<0.5em> (-19,0)};
{(-10,0) \ar@<-0.5em> (-19,0)};
{(10,0) \ar (19,0)};
{(19,0) \ar@<0.5em> (10,0)};
{(19,0) \ar@<-0.5em> (10,0)};
\endxy
\] 
This is a pure analog of a well-known recollement due to Krause~\cite{krause-stable derived cat of a Noetherian scheme}.

It is natural to ask if there are dual results, and proving this is the purpose of the present article. So throughout, we will let $\class{PP}$ denote the class of all pure-projective $R$-modules. These are precisely the modules that are direct summands (retracts) of direct sums of finitely presented modules. Our goal is to show that $K_{ac}(\class{PP})$, the chain homotopy category of all acyclic chain complexes of pure-projectives, is compactly generated, and fits into a recollement analogous to the one above. 

In Theorem~\ref{them-proj-cot-pair} we provide a general method for constructing abelian model structures on chain complexes, having certain complexes of pure-projectives as their cofibrant objects. By Hovey's correspondence these model structures are equivalent to ``pure-projective'' model structures on the exact category $\ch_{pur}$ --- This is the category of chain complexes but along with the proper class of short exact sequences that are \emph{pure exact in each degree}.
We apply Theorem~\ref{them-proj-cot-pair}  throughout the paper to construct three abelian model structures which induce the recollement we seek.  Below we describe these three constructions in more detail. To do so, we note first that given any $R$-module $M$, we let $S^n(M)$ denote the chain complex consisting only of $M$ in degree $n$ and 0 elsewhere.



\begin{enumerate}
\item In Sections~\ref{sec-pp} and~\ref{sec-proj-cot-pairs}  we take a different approach to reestablish some results that are due to Krause, Stovicek, and Emmanouil (see~\cite{krause-approximations and adjoints, stovicek-purity, emmanouil-pure-acyclic-complexes, emmanouil-relation-K-flatness-K-projectivity}). Our proofs are different though, making use of some techniques from~\cite[App.~A]{bravo-gillespie-hovey}.
We let $\dwclass{PP}$ denote the class of all complexes that are \emph{degreewise} pure-projective. That is, $P \in \dwclass{PP}$ just means $P$ is a chain complex of pure-projective modules. On the other hand, we let $\class{A}_{pur}$ denote the class of all pure acyclic complexes.  
 Generalizing techniques from~\cite[App.~A]{bravo-gillespie-hovey}, we give a hands on direct proof  that $(\dwclass{PP}, \class{A}_{pur})$ is a complete cotorsion pair in the exact category $\ch_{pur}$, by showing it to be cogenerated by the set of all bounded above complexes of finitely presented modules; see Theorem~\ref{thm-bounded above complexes of finitely presented cogenerate}. It follows that this cotorsion pair is in fact cogenerated by the set $\{S^n(M_i)\}$ where this set ranges through all integers $n$ and all finitely presented modules $M_i$. This gives us a finitely generated (and monoidal) abelain model structure on chain complexes whose homotopy category is $\class{D}_{pur}(R)$, the pure derived category of $R$. See Corollary~\ref{cor-pure-proj-model}.

\item In Section~\ref{sec-acyclic-pp} we turn to acyclic complexes of pure-projectives.  We let $\exclass{PP}$ denote the class of all acyclic  complexes of pure-projective $R$-modules (just \emph{exact} in the usual sense, not pure exact). We let $\class{V} = \rightperp{\exclass{PP}}$ denote the right $\Ext^1$-orthogonal of $\exclass{PP}$ with respect to the exact structure $\ch_{pur}$. Again using techniques from~\cite[App.~A]{bravo-gillespie-hovey}, we show directly that $(\exclass{PP}, \class{V})$ is a complete cotorsion pair in the exact category $\ch_{pur}$, by showing it to be cogenerated by a set. See
Theorem~\ref{theorem-acyclic complexes of pure-projectives}. We will see that the class $\class{V}$ determines a thick subcategory of $K(R)$, and  Corollary~\ref{cor-exPP-model} says that $(\exclass{PP}, \class{V})$ determines an abelian model structure on chain complexes whose homotopy category coincides with the Verdier quotient $K(R)/\class{V}$. Moreover, $K(R)/\class{V} \cong K_{ac}(\class{PP})$, the chain homotopy category of all acyclic complexes of pure-projective modules. It follows that $K(R)/\class{V} \cong K_{ac}(\class{PP})$ is a well-generated triangulated category. But as we explain below, by the end of the paper we will show it to in fact be a compactly generated category. 

Note that the class $\class{V}$ is a dual to the class of all K-flat complexes, but we do not provide a further study of the complexes in $\class{V}$. Emmanouil and Kaperonis will provide an in depth look at the complexes in the class $\class{V}$, which they call \emph{K-absolutely pure}. Indeed they are dual to the K-flat complexes in many ways that will convince the reader in their forthcoming preprint~\cite{emmanouil-kaperonis-K-flatness-pure}. 

Our work here also implies (see Theorem~\ref{theorem-derived-cat-K-pur}) that $\class{D}(R)$, the usual derived category of $R$, is equivalent to the chain homotopy category of all K-absolutely pure complexes with pure-projective components.  
\item In Section~\ref{sec-dg-pp} we construct a ``pure-projective'' variant of the standard projective model structure for $\class{D}(R)$. Recall that $\{S^n(R)\}$ cogenerates the standard projective model structure on $\ch$, where the cofibrant objects are the so-called DG-projective complexes. (See~\cite[Example~3.3]{hovey}.) However, if you let $\{S^n(R)\}$ cogenerate a complete cotorsion pair with respect to the exact category $\ch_{pur}$, then the result is an equivalent model structure whose cofibrant objects are what we call \emph{DG-pure-projective}. These are chain complexes $P$ for which each $P_n$ is a pure-projective $R$-module and for which any chain map $P \xrightarrow{} E$, to any exact complex $E$, is null homotopic. The main results about this model structure are stated in Corollary~\ref{cor-DG-pure-projective-model}.
\end{enumerate}

The three model structures described in (1), (2) and (3)  are interrelated in a simple way.
In particular, let $\class{E}$ denote the class of all acyclic chain complexes. This is the class of trivial objects in the DG-pure-projective model structure above in (3). The classes of cofibrant objects in (1) and (2) satisfy $\dwclass{PP}\cap \class{E} = \exclass{PP}$, and so the following desired result is deduced by applying the main result of~\cite{gillespie-recollement}.

\begin{theorem}[See Theorem~\ref{theorem-recollement-pp}]
We have a recollement of triangulated categories
\[
\xy
(-28,0)*+{K(R)/\class{V}};
(0,0)*+{\class{D}_{pur}(R)};
(25,0)*+{\class{D}(R)};
{(-19,0) \ar (-10,0)};
{(-10,0) \ar@<0.5em> (-19,0)};
{(-10,0) \ar@<-0.5em> (-19,0)};
{(10,0) \ar (19,0)};
{(19,0) \ar@<0.5em> (10,0)};
{(19,0) \ar@<-0.5em> (10,0)};
\endxy
\] which when restricting the first two categories to cofibrant objects becomes 
\[
\xy
(-28,0)*+{K_{ac}(\class{PP})};
(0,0)*+{K(\class{PP})};
(25,0)*+{\class{D}(R)};
{(-19,0) \ar (-10,0)};
{(-10,0) \ar@<0.5em> (-19,0)};
{(-10,0) \ar@<-0.5em> (-19,0)};
{(10,0) \ar (19,0)};
{(19,0) \ar@<0.5em> (10,0)};
{(19,0) \ar@<-0.5em> (10,0)};
\endxy
.\]  
In particular, the Verdier quotient $K(R)/\class{V} \cong  K_{ac}(\class{PP})$ must be compactly generated. 
\end{theorem}

Although the current impetus for this work has been the recent papers~\cite{emmanouil-K-flatness-and-orthogonality-in-homotopy-cats} and~\cite{gillespie-K-flat}, we note that the pure derived category has been studied from various perspectives by many authors; in particular we are aware of~\cite{hovey-christensen-relative hom alg, krause-approximations and adjoints, gillespie-G-derived, stovicek-purity, emmanouil-pure-acyclic-complexes}.  Our work also relates to seminal work on the homotopy category of projective modules, first studied in~\cite{jorgensen-homotopy-projectives, neeman-flat}, and from the model category point of view in~\cite{bravo-gillespie-hovey, becker}.
We note too that Holm and J\o rgensen were able to prove the compact generation, back in~\cite{holm-jorgensen-compactly-generated}, of several interesting homotopy categories of chain complexes. In fact they showed in~\cite[\S6 Item~(3)]{holm-jorgensen-compactly-generated} that $K(\class{PP})$ is compactly generated whenever the ring $R$ has finite (left) pure global dimension.

\section{Preliminaries}\label{sec-preliminaries}
Throughout the paper, $R$ denotes a ring with identity, $R$-Mod the category of left $R$-modules, and $\ch$ the category of chain complexes of left $R$-modules. We denote the opposite ring of $R$ by $R^\circ$, and so may write $R^\circ$-Mod to denote the category of right $R$-modules, and similar for $\chop$.
Our convention is that the differentials of our chain complexes lower degree, so $\cdots
\xrightarrow{} X_{n+1} \xrightarrow{d_{n+1}} X_{n} \xrightarrow{d_n}
X_{n-1} \xrightarrow{} \cdots$ is a chain complex.
The \emph{$n^{\text{th}}$
cycle module} is the module $Z_{n}X := \ker{d_{n}}$ and the
\emph{$n^{\text{th}}$ boundary module} is the module $B_{n}X := \im{d_{n+1}}$. The \emph{$n^{\text{th}}$ homology module} is defined to be
$H_{n}X := Z_{n}X/B_{n}X$. A complex $E$ is said to \emph{acyclic}, or \emph{exact} (we use these terms interchangably), if $H_nE = 0$ for each $n$.

Given a chain complex $X$, the
\emph{$n^{\text{th}}$ suspension of $X$}, denoted $\Sigma^n X$, is the complex given by
$(\Sigma^n X)_{k} = X_{k-n}$ and $(d_{\Sigma^n X})_{k} = (-1)^nd_{k-n}$.
For a given $R$-module $M$, we denote the \emph{$n$-disk on $M$} by $D^n(M)$. This is the complex consisting only of $M \xrightarrow{1_M} M$ concentrated in degrees $n$ and $n-1$, and 0 elsewhere. We denote the \emph{$n$-sphere on $M$} by $S^n(M)$, and this is the complex consisting only of $M$ in degree $n$ and 0 elsewhere. 

The chain homotopy category of $R$ is denoted $K(R)$. Recall that its objects are also chain complexes but its morphisms are chain homotopy classes of chain maps.

Given two chain complexes $X, Y \in \ch$, the total  $\Hom$ chain complex will be denoted by $\homcomplex(X,Y)$. It is the complex of abelian groups $$\cdots \xrightarrow{} \prod_{k \in
\Z} \Hom_R(X_{k},Y_{k+n}) \xrightarrow{\delta_{n}} \prod_{k \in \Z}
\Hom_R(X_{k},Y_{k+n-1}) \xrightarrow{} \cdots,$$ where $(\delta_{n}f)_{k}
= d_{k+n}f_{k} - (-1)^n f_{k-1}d_{k}$.
Its homology satisfies $H_n[Hom(X,Y)] = K(R)(X,\Sigma^{-n} Y)$.

We recall too the usual tensor product of chain complexes. Given $X \in \chop$ and $Y \in \ch$, their tensor product $X
\otimes_R Y$ is defined by $$(X \otimes_R Y)_n = \bigoplus_{i+j=n} (X_i
\otimes_R Y_j)$$ in degree $n$. The boundary map $\delta_n$ is defined
on the generators by the formula $\delta_n (x \otimes y) = dx \otimes y +
(-1)^{|x|} x \otimes dy$, where $|x|$ is the degree of the element
$x$.

Since this paper is considering chain complexes of modules over a ring, the reader may find the books~\cite{garcia-rozas, christensen-foxby-holm-book} to be helpful references. 

\subsection{Cotorsion pairs and abelian model structures}\label{subsection-cot-model}
Besides chain complexes, this paper heavily uses standard facts about \emph{cotorsion pairs} and \emph{abelian model categories}. Standard references for cotorsion pairs include~\cite{enochs-jenda-book} and~\cite{trlifaj-book} and the connection to abelian model categories can be found in~\cite{hovey},~\cite{gillespie-exact model structures}, and~\cite{gillespie-hereditary-abelian-models}. Basic language associated to \emph{exact categories}, in the sense of Quillen, from \cite{quillen-algebraic K-theory}, \cite{keller-exact-cats}, and~\cite{buhler-exact categories} will also be used. It may help the reader to know that, because of~\cite[Appendix~B]{gillespie-G-derived}, an exact structure (in the sense of Quillen) on an abelian category is the same thing as a proper class of short exact sequences (in the sense of Mac\,Lane). So on abelian categories, such as $\ch$, there is no difference between an exact model structure in the sense of~\cite{gillespie-exact model structures} and an abelian model structure compatible with a proper class of short exact sequences as in~\cite{hovey}. 

\subsection{Pure exact structures}\label{subsec-pure exact structures} 
Here we recall some language and notation that was used in~\cite{gillespie-K-flat}. We will be considering the category $R$-Mod of (left) $R$-modules along with the class $\class{P}ur$ of pure exact sequences. This gives us an exact structure $(R\textnormal{-Mod}, \class{P}ur)$ which we will denote by $R\textnormal{-Mod}_{pur}$. We let $\class{A}$ denote the class of all $R$-modules and $\class{PI}$ denote the class of all pure-injective $R$-modules. We have a complete cotorsion pair $(\class{A},\class{PI})$ in $R\textnormal{-Mod}_{pur}$. On the other hand, we let $\class{PP}$ denote the class of all pure-projective $R$-modules. These are characterized as the modules that are direct summands of direct sums of finitely presented modules. $(\class{PP},\class{A})$ is also a complete cotorsion pair in $R\textnormal{-Mod}_{pur}$.

The exact structure $R\textnormal{-Mod}_{pur}$ lifts to an exact structure on $\ch$, the category of chain complexes. We let $\ch_{pur}$ denote this exact category whose short exact sequences are pure exact sequences of $R$-modules in each degree. We will denote the associated Yoneda Ext group of all (equivalence classes of) degreewise pure short exact sequences by $\Ext^1_{pur}(X,Y)$. 

 The following proposition tells us that we can use $\Ext^1_{pur}$ to construct complete cotorsion pairs in $\ch_{pur}$ by the usual method of cogenerating by a set. To state it, let $\{M_i\}$ be a set of representatives of all the isomorphism classes of finitely presented $R$-modules. This is a generating set for the exact category $R\textnormal{-Mod}_{pur}$. By taking all their $n$-disks ($n \in \Z$), this lifts to a set $\{D^n(M_i)\}$, which is a set of generators for the exact category $\ch_{pur}$.

\begin{proposition}\label{prop-cogeneration-by-set}
Let $\class{S}$ be any set (not a proper class) of chain complexes. Then $(\leftperp{(\rightperp{\class{S}})}, \rightperp{\class{S}})$ is a functorially complete cotorsion pair with respect to the degreewise pure exact structure $\ch_{pur}$. Moreover, the class $\leftperp{(\rightperp{\class{S}})}$ consists precisely of direct summands (retracts) of transfinite degreewise pure extensions of complexes in $\class{S}\cup \{D^n(M_i)\}$. 
\end{proposition}

\begin{proof}
This is a special instance of~\cite[Prop.~5.4]{gillespie-G-derived} and the ``Remark 2'' that follows it. (However, the quoted proposition's statement and proof are sloppy there. First, Prop.~5.4~(2) of op.~cit. should say $\class{S}\cup\{G_i\}$, which translates for us here to be $\class{S}\cup \{D^n(M_i)\}$. Second, the functors denoted by $\textnormal{G-}\Ext^1_{\cha{G}}$ in the first paragraph of that proof should really be denoted 
$\textnormal{G-}\Ext^1_{\class{G}}$, as in the paragraph before Lemma~4.2 of op.~cit.) The category $\cha{G}_G$ in Remark 2 of op.~cit. is discussed in Section~4.1 of op.~cit. and yields $\ch_{pur}$ when the set of generators $\{G_i\}$ is a representative set of all finitely presented $R$-modules, which we have done above by taking $\{G_i\}= \{M_i\}$.  
\end{proof}


\section{Complexes of pure-projectives}\label{sec-pp}

Recall that we are letting $\class{PP}$ denote the class of all pure-projective $R$-modules. 
Throughout the paper, and especially throughout this section, we let $\dwclass{PP}$ denote the class of all complexes that are \emph{degreewise} pure-projective. That is, $P \in \dwclass{PP}$ just means $P$ is a chain complex of pure-projective modules. We let $\class{A}_{pur}$ denote the class of all pure acyclic complexes, as in~\cite[Def.~4.1]{gillespie-K-flat}. Using techniques from~\cite[Appendix~A]{bravo-gillespie-hovey}, we will give a direct proof that $(\dwclass{PP}, \class{A}_{pur})$ is a complete cotorsion pair in the exact category $\ch_{pur}$.

\begin{lemma}\label{lemma-bounded above complexes of pp}
Let $F$ be a bounded above complex of finitely presented modules. Then any chain map $f : F \xrightarrow{} Y$ with $Y \in \class{A}_{pur}$ is null homotopic. In other words, $\homcomplex(F,Y)$ is acyclic, or equivalently, $\Ext^1_{pur}(F,Y) = 0$ for all $Y \in \class{A}_{pur}$.
\end{lemma}

\begin{proof}
Assume we are given any chain map $f :
F\xrightarrow{}Y$ with $Y \in \class{A}_{pur}$.  We will show it is null homotopic by constructing a chain
homotopy $s_{n} : F_{n}\xrightarrow{}Y_{n+1}$ with
$ds_{n}+s_{n-1}d=f_{n}$, by downwards induction on $n$.  Since $F$ is
bounded above, we can take $s_{n}=0$ for large $n$ to begin the
induction.  So we suppose that $s_{i}$ has been defined for $i\geq n$
and that $d_{n+2}s_{n+1}+s_{n}d_{n+1}=f_{n+1}$.

We first modify $s_{n}$ to a new map $\tilde{s}_{n}$ so that
this identity still holds, and also construct a map $s_{n-1}$ such that
$d_{n+1}\tilde{s}_{n}+s_{n-1}d_n=f_{n}$. Note first that
\[
(f_{n}-ds_{n})d = d (f_{n+1}-s_{n}d)=d^{2}s_{n+1}=0,
\]
so there is an induced map $g_{n} :
F_{n}/B_{n}F\xrightarrow{}Y_{n}$, such that the composite $F_n \xrightarrow{\pi_n} F_{n}/B_{n}F\xrightarrow{g_n} Y_{n}$ equals $f_n - d_{n+1}s_n$.  Now consider the bounded complex
$X$ of finitely presented modules $$\cdots \xrightarrow{} 0 \xrightarrow{} F_{n}/B_{n}F \xrightarrow{\bar{d}_n}
 F_{n-1} \xrightarrow{\pi_{n-1}} F_{n-1}/B_{n-1}F \xrightarrow{} 0 \xrightarrow{} \cdots$$ where degree $n$ is $X_{n}=F_{n}/B_{n}F$.  
 There is a chain map $g : X\xrightarrow{}Y$ that
is $g_{n}$ in degree $n$, $f_{n-1}$ in degree $n-1$, and $f_{n-2}\bar{d}_{n-1}$ in
degree $n-2$.  It is easy to show that \emph{any} chain map from a bounded complex of finitely presented modules to a pure acyclic complex must be null homotopic. (See~\cite[Lemma~A.4]{bravo-gillespie-hovey} for a short proof.) In particular, $g$ is null homotopic.  This gives us maps $s_{n}' :
F_{n}/B_{n}F\xrightarrow{}Y_{n+1}$ and $s_{n-1} :
F_{n-1}\xrightarrow{}Y_{n}$ such that
$d_{n+1}s_{n}'+s_{n-1}\bar{d}_n= g_n$.  Composing with $F_n \xrightarrow{\pi_n} F_{n}/B_{n}F$ this becomes
$$d_{n+1}(s_{n}'\pi_n)+s_{n-1}d_n= f_n - d_{n+1}s_n$$
and so then setting $\tilde{s}_{n} = s_{n}+s_{n}'\pi_n$ we see the required equation
$$d_{n+1}(s_{n}+s_{n}'\pi_n)+s_{n-1}d_n =f_{n}.$$
Moreover, with $\tilde{s}_{n}= s_{n}+s_{n}'\pi_n$ substituted for $s_n$, we see that the original identity 
$$
d_{n+2}s_{n+1} + (s_{n}+s'_{n}\pi_n)d_{n+1}=f_{n+1}
$$
still holds because $s'_{n}\pi_nd_{n+1} = 0$.
\end{proof}

By an \textbf{fp-free} module we mean a direct sum of finitely presented modules. Note that pure-projective modules are precisely the direct summands of fp-free modules. 

\begin{lemma}\label{lemma-complexes of pure-projectives are retracts of complexes of fp-frees}
Let $P$ be a complex of pure-projective $R$-modules. Then
$P$ is a direct summand of a chain complex $F$ of fp-free modules. That is, each component $F_n$ is fp-free.  Furthermore, if $P$
is acyclic then $F$ too can be taken to be acyclic.
\end{lemma}

\begin{proof}
Recall Eilenberg's swindle (Corollary~2.7 of~\cite{lam}) allows one to
construct, for any projective module $P$ a free module $F$ such that
$P \oplus F \cong F$. The same trick allows one to
construct, for any pure-projective module $P$ an fp-free module $F$ such that
$P \oplus F \cong F$. So given any complex of pure-projectives $P$, we can find for each $P_n$ an fp-free $F_n$ such that $P_n \oplus F_n
\cong F_n$. Then $P \oplus (\bigoplus_{n \in \Z} D^n(F_n))$ is a complex
of fp-free modules. Indeed in degree $n$ the complex equals $P_n \oplus
F_n \oplus F_{n+1} \cong F_n \oplus F_{n+1}$ which is fp-free. Of course
$P$ is a direct summand of $P \oplus (\bigoplus_{n \in \Z} D^n(F_n))$ by
construction and also $P \oplus (\bigoplus_{n \in \Z} D^n(F_n))$ is acyclic
whenever $P$ is acyclic.
\end{proof}

In the following we reach our goal: Using the techniques from~\cite[Appendix~A]{bravo-gillespie-hovey} to show that $(\dwclass{PP}, \class{A}_{pur})$ is a complete cotorsion pair in the exact category $\ch_{pur}$. We show here that it is cogenerated by the class of all bounded above complexes of finitely presented modules, which of course is skeletally small. In fact, it is cogenerated by the set $\{S^n(M_i)\}$ of all spheres on finitely generated modules $M_i$.

\begin{theorem}\label{thm-bounded above complexes of finitely presented cogenerate}
$(\dwclass{PP}, \class{A}_{pur})$ is a complete cotorsion pair in the exact category $\ch_{pur}$. It is
cogenerated by the collection of all bounded above complexes of finitely presented modules. 

In fact, the cotorsion pair is cogenerated by the set $\class{T} = \{S^n(M_i)\}$ where $\{M_i\}$ is any representative set of all the isomorphism classes of finitely presented $R$-modules. It follows that any chain complex of pure-projective modules is a direct summand of a transfinite degreewise split extension of $\class{T} \cup \{D^n(M_i)\}$.
\end{theorem}

\begin{proof}
 Let $S$ be a set of representatives for all the isomorphism classes of bounded above complexes of finitely presented modules. We let $\class{T}$ be as described, and note that $\class{T} \cup \{D^n(M_i)\} \subseteq \class{S}$.  By Propostion~\ref{prop-cogeneration-by-set}, the sets $\class{S}$ and $\class{T}$ each cogenerate functorially complete cotorsion pairs $(\leftperp{(\rightperp{\class{S}})}, \rightperp{\class{S}})$ and $(\leftperp{(\rightperp{\class{T}})}, \rightperp{\class{T}})$, in the exact category $\ch_{pur}$. Moreover, the class $\leftperp{(\rightperp{\class{S}})}$ (resp. $\leftperp{(\rightperp{\class{T}})}$) consists precisely of direct summands of transfinite degreewise pure extensions of complexes in $\class{S}$ (resp.  $\class{T} \cup \{D^n(M_i)\}$). We will show that $$(\leftperp{(\rightperp{\class{S}})}, \rightperp{\class{S}}) = (\leftperp{(\rightperp{\class{T}})}, \rightperp{\class{T}}) = (\dwclass{PP}, \class{A}_{pur}).$$

First we will show $\rightperp{\class{S}} = \rightperp{\class{T}} = \class{A}_{pur}$. We start by noting $\rightperp{\class{T}} = \class{A}_{pur}$. Indeed $X \in \rightperp{\class{T}}$ if and only if for all integers $n$, and all the finitely presented modules $M_i$, we have $$0 = \Ext^1_{pur}(S^n(M_i),X) = \Ext^1_{dw}(S^n(M_i),X) \cong  H_{n-1}\Hom_R(M_i,X).$$ This means $X \in \rightperp{\class{T}}$ if and only if each $\Hom_R(M_i,X)$ is acyclic which means that $X$ is a pure acyclic. So $\rightperp{\class{T}} = \class{A}_{pur}$. Next, since $\class{T} \subseteq \class{S}$, we get $\rightperp{\class{S}} \subseteq \rightperp{\class{T}}$. So $\rightperp{\class{S}} \subseteq \class{A}_{pur}$. On the other hand, we have $\class{A}_{pur}\subseteq \rightperp{\class{S}}$, by Lemma~\ref{lemma-bounded above complexes of pp}. So we have shown $\rightperp{\class{S}} = \class{A}_{pur}$ too.

It follows that $(\leftperp{(\rightperp{\class{S}})}, \rightperp{\class{S}}) = (\leftperp{(\rightperp{\class{T}})}, \rightperp{\class{T}}) =  (\leftperp{\class{A}_{pur}}, \class{A}_{pur})$. To complete the proof, it is enough to show $\leftperp{(\rightperp{\class{S}})} = \dwclass{PP}$. The containment $\leftperp{(\rightperp{\class{S}})} \subseteq \dwclass{PP}$ follows from the containment $\class{S} \subseteq \dwclass{PP}$ and the fact that $\dwclass{PP}$ is closed under direct summands and transfinite degreewise pure extensions. One way to see this is to observe that since $\class{PP}$ is the left hand side of a cotorsion pair in $R\textnormal{-Mod}_{pur}$, the class $\class{PP}$ must be closed under direct summands and under transfinite pure extensions. 

To show $\dwclass{PP} \subseteq \leftperp{(\rightperp{\class{S}})}$, we will show that every $P \in \dwclass{PP}$ is a direct summand of a transfinite degreewise pure extension of complexes in $\class{S}$. By Lemma~\ref{lemma-complexes of pure-projectives are retracts of complexes of fp-frees} we only need to show that any complex $F$ of fp-free modules (again this means that each $F_n$ is a direct sum of finitely presented modules) is a transfinite degreewise pure extension of bounded above complexes of finitely presented modules. But this can be done by imitating the argument in~\cite[Theorem~A.3]{bravo-gillespie-hovey}. Let us give the argument for the readers convenience. 

 So let $F$ be a complex of fp-free
modules and write each $F_n = \bigoplus_{i \in I_n} M_i$ for some collection $\{M_i\}_{i\in I_n}$
of finitely presented modules.  Assuming $F$ is nonzero we can find
a nonzero $F_n$ and we take one nonzero summand $M_j$ for some $j \in
I_n$. We start to build a bounded above subcomplex $X \subseteq F$ by
setting $X_n = M_j$ and setting $X_i = 0$ for all $i > n$. Now note
$d(M_j) \subseteq \bigoplus_{i \in I_{n-1}} M_i$ and that we can find a \emph{finite} 
subset $L^0_{n-1}\subseteq I_{n-1}$ such that $d(M_j) \subseteq \bigoplus_{i \in L^0_{n-1}}
M_i$. We set $X_{n-1} = \bigoplus_{i \in L^0_{n-1}} M_i$.  We can continue down in the same way
finding $L^0_{n-2} \subseteq I_{n-2}$ with $|L^0_{n-2}|$ finite and with
$d(\bigoplus_{i \in L^0_{n-1}} M_i) \subseteq \bigoplus_{i \in L^0_{n-2}}
M_i$. In this way we continue to get a subcomplex of $X$: 
\[
X^0 = \cdots
\xrightarrow{} 0 \xrightarrow{} M_j \xrightarrow{} \bigoplus_{i \in
L^0_{n-1}} M_i \xrightarrow{} \bigoplus_{i \in L^0_{n-2}} M_i \xrightarrow{}
\cdots 
\]
Since finitely presented modules are closed under finite direct sums, the complex $X^0$ is a nonzero bounded above complex of finitely
presented modules.

Note that $F/X^0$ is fg-free in each degree, and assuming it is nonzero we can in turn find
another nonzero subcomplex $X^1/X^0 \subseteq F/X^0$ with $X^1/X^0$ a bounded above complex of finitely presented modules
and with its quotient
$$(F/X^0)/(X^1/X^0) \cong F/X^1$$ a complex of fp-free modules. Note that we can
identify these quotients, in particular $F/X^0$ and $F/X^1$, with complexes whose degree $n$
entry is $\bigoplus_{i \in I_n-L^0_n} M_i$ (resp. $\bigoplus_{i \in I_n-L^1_n} M_i$) and in doing so we may
continue to find an increasing union $0 \neq X^0 \subseteq X^1
\subseteq X^2 \subseteq \cdots $ of degreewise split extensions corresponding to a nested union of
subsets $L^0_n \subseteq L^1_n \subseteq L^2_n \subseteq \cdots$ for
each $n$. Assuming this process doesn't terminate at a finite step, we set $X^{\omega} =
\cup_{\alpha < \omega} X^{\alpha}$ and note that $X^{\omega}_n =
\bigoplus_{i \in L^{\omega}_n} M_i$ where $L^{\omega}_n =
\cup_{\alpha < \omega} L^{\alpha}_n$. So still, $X^{\omega}$ and
$F/X^{\omega}$ are complexes of fp-free modules. Therefore we
can continue this process with $F/X^{\omega}$ to obtain $X^{\omega
+1}$ with all the properties we desire. Using this process we can
obtain an ordinal $\lambda$ and a continuous union $F = \cup_{\alpha <
\lambda} X^{\alpha}$ of degreewise split extensions with each $X_{\alpha}, X_{\alpha + 1}/X_{\alpha}$
a complex of fp-frees and with $X_0$ and each $X_{\alpha +1}/X_{\alpha}$ a bounded above complex of finitely presented modules.
Therefore we have shown $\dwclass{PP} \subseteq \leftperp{(\rightperp{\class{S}})}$.

This completes the proof that $(\leftperp{(\rightperp{\class{S}})}, \rightperp{\class{S}}) = (\leftperp{(\rightperp{\class{T}})}, \rightperp{\class{T}}) = (\dwclass{PP}, \class{A}_{pur})$  is a (functorially) complete cotorsion pair in the exact category $\ch_{pur}$. 

It follows from Proposition~\ref{prop-cogeneration-by-set} that any chain complex of pure-projective modules is a direct summand of a transfinite degreewise pure extension of the set $\class{T} \cup \{D^n(M_i)\}$. But all complexes is $\class{T} \cup \{D^n(M_i)\}$ have pure-projective components, so any such extension must actually be a degreewise \emph{split} extension. 
\end{proof}

\begin{corollary}[Stovicek~\cite{stovicek-purity}]\label{cor-pp-contractible}
Any pure acyclic complex of pure-projective $R$-modules is contractible. 
\end{corollary}

\begin{proof}
$Y \in \class{A}_{pur}$ if and only if for all $P \in \dwclass{PP}$ we have $$\Ext^1_{pur}(P,\Sigma^n Y) =  \Ext^1_{dw}(P,\Sigma^n Y) =  0.$$ By a well-known fact, for example see~\cite[Lemma~2.1]{gillespie}, this is equivalent to the statement that $\homcomplex(P,Y)$ is acyclic for all $P \in \dwclass{PP}$. So now if $X \in \dwclass{PP}\cap\class{A}_{pur}$ then it follows that $\homcomplex(X,X)$ is acyclic. It means all chain maps $X \xrightarrow{} X$ are null homotopic, and so $X$ is contractible. 
\end{proof}

\section{Complexes of pure-projectives and abelian model structures}\label{sec-proj-cot-pairs}

Let $\class{S}$ be any set (not just a proper class) of chain complexes of pure-projective modules, and closed under suspensions. We set $\class{V}_{\class{S}} := \rightperp{\class{S}}$, where this $\Ext^1$-orthogonal  is understood to be taken in the exact category $\ch_{pur}$ of Section~\ref{subsec-pure exact structures}. Finally, set $\class{C}_{\class{S}} := \leftperp{\class{V}_{\class{S}}}$, again in $\ch_{pur}$. The point of this section is to prove the following theorem which will be a tool applied, not just to the cotorsion pair 
$$(\class{C}_{\class{T}}, \class{V}_{\class{T}})  = (\dwclass{PP},\class{A}_{pur})$$ of Theorem~\ref{thm-bounded above complexes of finitely presented cogenerate}, but throughout the paper. 

 \begin{theorem}\label{them-proj-cot-pair} 
$(\class{C}_{\class{S}}, \class{V}_{\class{S}})$   is a projective cotorsion pair in the exact category $\ch_{pur}$. By this we mean that the triple $\mathfrak{M} = (\class{C}_{\class{S}}, \class{V}_{\class{S}}, All)$ is an abelian (equivalently, exact) model structure with respect to the degreewise pure exact structure. 

Moreover, each of the following hold:
\begin{enumerate}
\item $\class{C}_{\class{S}}$ is a suspension closed class of complexes of pure-projectives, and the thick class $\class{V}_{\class{S}}$ is closed under suspensions and contains all contractible complexes. A complex $X$ is in $\class{V}_{\class{S}}$ if and only if $\homcomplex(C,X)$ is acyclic for all $C \in \class{C}_{\class{S}}$ (or just all $C \in \class{S}$). Equivalently, any chain map $C \rightarrow{} X$ is null homotopic whenever $C\in\class{C}_{\class{S}}$ (or just $C \in \class{S}$).
\item If $f : X \xrightarrow{} Y$ is a homotopy equivalence, then $X \in \class{V}_{\class{S}}$ if and only if $Y \in \class{V}_{\class{S}}$. In fact, $\class{V}_{\class{S}}$ is a thick subcategory of the homotopy category $K(R)$. A chain map $f : X \xrightarrow{} Y$  is a weak equivalence in $\mathfrak{M}$ if and only if its cone $C(f) \in \class{V}_{\class{S}}$.
 \item The homotopy category of $\mathfrak{M}$ satisfies $\textnormal{Ho}(\mathfrak{M}) \cong K(R)/\class{V}_{\class{S}}$. That is, $\textnormal{Ho}(\mathfrak{M})$ is equivalent to the Verdier quotient. The model structure shows these categories to be equivalent to $K(\class{C}_{\class{S}})$, where $K(\class{C}_{\class{S}})$ is the full subcategory of $K(R)$ consisting of all complexes (homotopy equivalent to one) in $\class{C}_{\class{S}}$. 
\item $\mathfrak{M}$ is cofibrantly generated and so $\textnormal{Ho}(\mathfrak{M}) =K(R)/\class{V}_{\class{S}}$ is a well-generated triangulated category. 
\item If each complex in $\class{S}$ is a bounded complex of finitely presented modules, then $\mathfrak{M}$ is finitely generated in the sense of~\cite[Section~7.4]{hovey-model-categories} and so $\textnormal{Ho}(\mathfrak{M})$ is a compactly generated triangulated category in this case. 
\item Assume the ring $R$ is commutative. If $\class{C}_{\class{S}}$ contains $S^0(R)$ and is closed under the usual tensor product of chain complexes, then $\mathfrak{M}$ is a monoidal model structure with respect to the tensor product. 
\end{enumerate}
 \end{theorem}

 \begin{proof}
 We have at once from Proposition~\ref{prop-cogeneration-by-set} that $$(\leftperp{(\rightperp{\class{S}})}, \rightperp{\class{S}}) = (\class{C}_{\class{S}}, \class{V}_{\class{S}})$$ is a functorially complete cotorsion pair with respect to the degreewise pure exact structure $\ch_{pur}$. Moreover, the class $\class{C}_{\class{S}}$ consists precisely of direct summands (retracts) of transfinite degreewise pure extensions of complexes in $\class{S}\cup \{D^n(M_i)\}$. It follows that every complex $C \in \class{C}_{\class{S}}$ is a complex of pure-projectives, since the class $\class{PP}$ is closed under direct 
summands and transfinite pure extensions. 

To show that $\mathfrak{M} = (\class{C}_{\class{S}}, \class{V}_{\class{S}}, All)$ is an abelian model structure means to show: (i) $(\class{C}_{\class{S}}, \class{V}_{\class{S}})$ is a complete cotorsion pair (Done!), (ii) $(\class{C}_{\class{S}}\cap \class{V}_{\class{S}}, All)$ is a complete cotorsion pair, and (iii) $\class{V}_{\class{S}}$ is a thick class. (All with repect to the exact structure $\ch_{pur}$.) Before showing (ii) we will show $\class{V}_{\class{S}}$ is thick and satisfies the properties in (1). 

The class $\class{V}_{\class{S}} := \rightperp{\class{S}}$ is closed under suspensions, since $\class{S}$ is assumed to be. We have that $X \in 
\class{V}_{\class{S}}$ if and only if for all $C \in \class{C}_{\class{S}}$ (or just all $C\in\class{S}$) we have $\Ext^1_{pur}(C,\Sigma^n X) =  \Ext^1_{dw}(C,\Sigma^n X) =  0$. By Lemma~\cite[Lemma~2.1]{gillespie} this is equivalent to the statement that $\homcomplex(C,X)$ is acyclic for all $C \in \cat{C}$ (or just all $C\in\class{S}$).
It follows that $\class{V}_{\class{S}}$ is thick, 
for suppose we have a degreewise pure short exact sequence
\[
0 \xrightarrow{} X \xrightarrow{} Y \xrightarrow{} Z \xrightarrow{} 0,
\]
where two out of three of the complexes are in $\class{V}_{\class{S}}$.  Now suppose
$C\in \cat{C}$.  Since $C$ is a complex of pure-projectives, the resulting
sequence
\[
0 \xrightarrow{} \homcomplex (C,X) \xrightarrow{} \homcomplex (C,Y) \xrightarrow{}
\homcomplex (C,Z)\xrightarrow{} 0
\]
is still short exact.  Since two out of three of these complexes are
acyclic, so is the third. Moreover, $\class{V}_{\class{S}} := \rightperp{\class{S}}$ is closed under direct summands, so $\class{V}_{\class{S}}$ is a thick class.  It is also now clear that if $X$ is contractible, then  $X\in \class{V}_{\class{S}}$.  

We have shown the properties in (1) , but to obtain the model structure we still need to see why $(\class{C}_{\class{S}}\cap \class{V}_{\class{S}}, All)$ is a complete cotorsion pair. Note that this is equivalent to showing that $\class{C}_{\class{S}}\cap \class{V}_{\class{S}}$ is precisely the class of projectives in the exact category $\ch_{pur}$. But note that if $X \in \class{C}_{\class{S}}\cap \class{V}_{\class{S}}$, then $X \xrightarrow{1_X} X$ is null homotopic, which means $X$ is a contractible complex with pure-projective components. It follows from~\cite[Lemma~4.5]{gillespie-G-derived} (its the special case of when $G = \oplus_i M_i$ is the direct sum of all finitely presented modules) that $\class{C}_{\class{S}}\cap \class{V}_{\class{S}} = \widetilde{\class{PP}}$, the class of all categorical projectives in $\ch_{pur}$. Moreover, $(\widetilde{\class{PP}}, All)$ is a complete cotorsion pair cogenerated by the set $\{D^n(M_i)\}$ of all disks on finitely presented modules. This completes the proof that $\mathfrak{M} = (\class{C}_{\class{S}}, \class{V}_{\class{S}}, All)$ is an abelian model structure with respect to $\ch_{pur}$, and satisfies the properties in (1).

We prove (2). Let $f : X \xrightarrow{} Y$ be a homotopy equivalence. Then it factors as $f=pi$ where $i$ is a degreewise split monomorphism with contractible cokernel and $p$ is a degreewise split epimorphism with contractible kernel. Since $\class{V}_{\class{S}}$ is a thick class in $\ch_{pur}$ and contains all contractible complexes, it follows that $X \in \class{V}_{\class{S}}$ if and only if $Y \in \class{V}_{\class{S}}$. Next, recall that a thick subcategory of $K(R)$ is a strictly full subcategory that is closed under suspensions, mapping cones and direct summands~\cite[Definitions~1.5.1 and~2.1.6]{neeman-book}. We already know $\class{V}_{\class{S}}$ is closed under suspensions and direct summands (since up to a homotopy equivalence, direct sums in $K(R)$ agree with direct sums in $\ch$; for example see~\cite[Lemma~2.2]{gillespie-frontiers-china}). To see that $\class{V}_{\class{S}}$ is closed under taking mapping cones, recall that for any chain map $f : X \xrightarrow{} Y$ there is a degreewise split short exact sequence $$0 \xrightarrow{} Y \xrightarrow{} C(f) \xrightarrow{} \Sigma X \xrightarrow{} 0.$$ In particular,  $C(f)$ is just an extension in $\ch_{pur}$ of $Y$ and $\Sigma X$, and so $C(f) \in \class{V}_{\class{S}}$ whenever $X,Y\in\class{V}_{\class{S}}$. Therefore, $\class{V}_{\class{S}}$ is a thick subcategory of $K(R)$.

To complete the proof of (2), we must show that a chain map $f : X \xrightarrow{} Y$  is a weak equivalence in $\mathfrak{M}$ if and only if its cone $C(f) \in \class{V}_{\class{S}}$. First, we note that any cone $C(f)$ can be constructed by pushout along the suspension sequence $0 \xrightarrow{} X \xrightarrow{} \oplus D^{n+1}(X_n) \xrightarrow{} \Sigma X \xrightarrow{} 0$, and $\oplus D^{n+1}(X_n)$ is a contractible complex.  Next, we recall that a degreewise pure monomorphism $i$ is a weak equivalence in $\mathfrak{M}$ if and only if its cokernel $\cok{i} \in \class{V}_{\class{S}}$, by~\cite[Lemma~5.8]{hovey}.  From the pushout construction, one can check that the cone $C(i)$ is in $\class{V}_{\class{S}}$ if and only if $\cok{i}$ is in $\class{V}_{\class{S}}$. A similar statement holds for any degreewise pure epimorphism $p$, by using~\cite[Dual of Prop.~2.15]{buhler-exact categories}.
Now by the factorization axiom, any morphism $f$ factors as $f=pi$ where $i$ is a degreewise pure monomorphism (with cokernel in $\class{C}_{\class{S}}$) and $p$ is a degreewise pure epimorphism with kernel  $\ker{p} \in \class{V}_{\class{S}}$. By the 2 out of 3 axiom, we have that $f$ is a weak equivalence if and only if $\cok{i} \in \class{V}_{\class{S}}$ too. 
By the octahedral axiom, there is an exact triangle $C(i) \xrightarrow{} C(f) \xrightarrow{} C(p) \xrightarrow{} \Sigma C(i)$ in  $K(R)$. Since exact triangles arise from short exact sequences, two terms of the triangle are in $\class{V}_{\class{S}}$ if and  only if the third term is too. So since $C(p), \ker{p} \in \class{V}_{\class{S}}$, we have $C(f) \in \class{V}_{\class{S}}$ if and only if $C(i), \cok{i} \in \class{V}_{\class{S}}$, which as already noted occurs if and only if $f$ is a weak equivalence. 

For statements (3) and (4), we refer the reader to proofs of similar statements. For the Verdier localization characterization, we see a similar argument in~\cite[Theorem~4.4]{gillespie-K-flat}. That the model structure is cofibrantly generated and the homotopy category is equivalent to $K(\class{C}_{\class{S}})$, see the arguments in the proof of~\cite[Theorem~3.9]{gillespie-K-flat}. It follows from~\cite{rosicky-brown representability combinatorial model srucs} that the homotopy category is well-generated.

We prove (5). We are claiming that the model structure is finitely generated whenever $\class{S}$ consists of bounded complexes of finitely presented modules. In this case, one can find for any $S \in \class{S}$, a short exact sequence in $\ch_{pur}$
$$0 \xrightarrow{} K_S \xrightarrow{i_S} P_S \xrightarrow{} S \xrightarrow{} 0$$
where $P_S = \oplus D^n(M_i)$ is a finite direct sum of disks on finitely presented modules $M_i$. Then $P_S$ is a finitely generated projective object of $\ch_{pur}$ and $K_S$ is also finitely generated, since it is bounded and each component must be finitely generated.  So the domains and codomains of the maps $i_S$ are \emph{finite relative to cofibrations}, and so we have a set $I = \{i_S\}$ of generating cofibrations, showing that the model structure is \emph{fiitely generated} in the sense of~\cite[Section~7.4]{hovey-model-categories}. In particular, $\textnormal{Ho}(\mathfrak{M})$ is compactly generated by~\cite[Corollary~7.4.4]{hovey-model-categories}.

Finally, for (6), assume $R$ is a commutative ring. We claim that the model structure is monoidal with respect to the usual tensor product of chain complexes. We refer the reader to the proof of~\cite[Theorem~5.1]{gillespie-K-flat}, which is similar in that it checks conditions (a), (b), (c), and (d) of~\cite[Theorem~7.2]{hovey} for an abelian model structure on $\ch_{pur}$. Our hypothesis that $\class{C}_{\class{S}}$ is closed under tensor products is the condition (b) and our hypothesis that $S^0(R) \in \class{C}_{\class{S}}$ is the condition (d).  The remaining conditions (a) and (c) are as in~\cite[Theorem~5.1]{gillespie-K-flat}. 
For condition (c), we emphasize that $X \otimes_R Y$ is always pure acyclic whenever $Y$ is pure acyclic. So if moreover $X$ and $Y$ are also both in $\dwclass{PP}$ then we will have $X \otimes_R Y \in \dwclass{PP}\cap\class{A}_{pur}$. By Corollary~\ref{cor-pp-contractible} this implies  $X \otimes_R Y$ is contractible and so it is in $\class{V}_{\class{S}}$. 
 \end{proof}

For our first application of Theorem~\ref{them-proj-cot-pair} we take the set $\class{T} = \{S^n(M_i)\}$ from Theorem~\ref{thm-bounded above complexes of finitely presented cogenerate}. In this case, $(\class{C}_{\class{T}}, \class{V}_{\class{T}}) = (\dwclass{PP}, \class{A}_{pur})$, and we are lead to the following corollary. We note that the essentials here have already been shown in different ways by combining results of Krause~\cite{krause-approximations and adjoints} and Stovicek in~\cite{stovicek-purity}.

\begin{corollary}\label{cor-pure-proj-model}
The triple $(\dwclass{PP}, \class{A}_{pur}, All)$ is a abelian model structure on the exact category $\ch_{pur}$ and satisfies each of the following properties:
\begin{enumerate}
\item Its homotopy category recovers $\class{D}_{pur}(R) = K(R)/\class{A}_{pur}$, the pure derived category of $R$, and the model structure provides a canonical equivalence $\class{D}_{pur}(R) \cong K(\class{PP})$, where $K(\class{PP})$ is the full subcategory of $K(R)$ consisting of all complexes (homotopy equivalent to one) in $\dwclass{PP}$. 
\item $\class{D}_{pur}(R)$ is a compactly generated triangulated category. In fact, the model structure is finitely generated in the sense of~\cite[Section~7.4]{hovey-model-categories}. 
\item In the case that $R$ is commutative, the model structure is monoidal with respect to the usual tensor product of chain complexes. 
\end{enumerate}
Moreover, the set $\class{T} = \{S^n(M_i)\}$ of Theorem~\ref{thm-bounded above complexes of finitely presented cogenerate} is a set of compact weak genrators for $\class{D}_{pur}(R)$, and the compact objects are the categorically finitely presented chain complexes.
\end{corollary}
 
 \begin{proof}
For the numbered statements, it is only left to comment as to why the model structure is monoidal, assuming $R$ is a commutative ring.  By the monoidal criteria provided in   Theorem~\ref{them-proj-cot-pair} we only need to show that $X \tensor_R Y \in \dwclass{PP}$ whenever $X,Y \in \dwclass{PP}$. To see this, we first note that the tensor product $M \otimes_R N$ of two $R$-modules is again an $R$-module and is finitely presented whenever $M$ and $N$ are both finitely presented. (For a nice proof of this, see Lemma~10.12.14 of the ``Stacks Project''.) Since tensor products commute with direct sums, and since the class $\class{PP}$ of pure-projective modules is closed under direct sums and direct summands it follows that $\class{PP}$ is closed under tensor products. Thus $\dwclass{PP}$ is closed under (chain complex) tensor products, as desired.
 
Let us prove the final statement about the compact objects.  Again, $\class{T} = \{S^n(M_i)\}$ where $\{M_i\}$ is some representative set for all the isomorphism classes of finitely presented $R$-modules. Theorem~\ref{thm-bounded above complexes of finitely presented cogenerate} showed that $\class{T}$ cogenerates the complete cotorsion pair $(\dwclass{PP}, \class{A}_{pur})$, in the exact category $\ch_{pur}$. This corresponds to the fact that $\class{T}$ is a set of compact weak generators for $\class{D}_{pur}(R)$. This is made explicit in~\cite[Remark~1, pp.~388]{gillespie-G-derived}. From the general theory of triangulated categories, it follows that the full subcategory of all compact objects of $\class{D}_{pur}(R)$ equals the smallest thick subcategory  containing $\class{T}$. We see that this corresponds to the chain homotopy category of all bounded complexes of finitely presented modules. This aligns with what Krause already found in~\cite[Corollary~4.6]{krause-approximations and adjoints}. Note that by~\cite[Lemma~4.1.1]{garcia-rozas}, this says that the compact objects of $\class{D}_{pur}(R)$ correspond to the \emph{categorically} finitely presented chain complexes. 
 \end{proof}
 
\section{Acyclic complexes of pure-projectives}\label{sec-acyclic-pp}

 Throughout this section, we let $\exclass{PP}$ denote the class of all acyclic (\emph{exact}) complexes of pure-projective $R$-modules. We let $\class{V} = \rightperp{\exclass{PP}}$ denote the right $\Ext^1$-orthogonal of $\exclass{PP}$ in $\ch_{pur}$. We will now turn to our second application of Theorem~\ref{them-proj-cot-pair}. We will find a set $\class{S}$ of complexes which cogenerates a complete cotorsion pair $$(\class{C}_{\class{S}}, \class{V}_{\class{S}}) =  (\exclass{PP},\class{V})$$ in the exact category $\ch_{pur}$. 
 We note that $\class{V}$ is precisely the class of \emph{K-absolutely pure} complexes studied in~\cite{emmanouil-kaperonis-K-flatness-pure}. 

\begin{proposition}\label{prop-filtrations for complexes of pure-projectives}
Let  $\kappa > \text{max}\{\, |R| \,  , \, \omega \,\}$ be a regular cardinal. Let $F$ be a nonzero acyclic complex of fp-free modules. That is, each component $F_n = \oplus_{i \in I_n} M_i$ is a direct sum of finitely presented modules $M_i$. Then $F$ is a transfinite degreewise split extension (continuous union) $F = \cup_{\alpha < \lambda}
Q_{\alpha}$ where each $Q_{\alpha}, Q_{\alpha + 1}/Q_{\alpha}$ are
also acyclic complexes of fp-free modules and such that $|Q_0|,
|Q_{\alpha + 1}/Q_{\alpha}| < \kappa$.
\end{proposition}

\begin{proof}
Note that $|M_i| < \kappa$ since each $M_i$ is finitely presented. 
So to start, we can use  the ``Exact Covering Lemma''~\cite[Lemma~7.2]{bravo-gillespie-hovey}  to find a nonzero $Q^0 \subseteq F$ of the
form $Q^0_n = \oplus_{i \in L^0_n} M_i$ for some subcollections
$L^0_n \subseteq I_n$ having $|L^0_n| < \kappa$ and such that $Q^0$ is still acyclic. Note that $Q^0$ and $F/Q^0$ are each also
acyclic complexes of fp-free modules, and $$0 \xrightarrow{} Q^0
\xrightarrow{} F \xrightarrow{} F/Q^0 \xrightarrow{} 0$$ is a
degreewise split short exact sequence. In particular, this is a short exact sequence in $\ch_{pur}$. So if it happens that $F/Q^0$ is nonzero we can in turn find
another nonzero subcomplex $Q^1/Q^0 \subseteq F/Q^0$ with $Q^1/Q^0$ and
$(F/Q^0)/(Q^1/Q^0) \cong F/Q^1$ both still acyclic complexes
of fp-free modules, and with $|Q^1/Q^0| < \kappa$. Note that as in the proof of Theorem~\ref{thm-bounded above complexes of finitely presented cogenerate}, we can
identify these quotients, such as $F/Q^0$, with the complexes whose degree $n$
entry is $\oplus_{i \in I_n-L^0_n} M_i$.  By repeatedly applying the ``Exact Covering Lemma''~\cite[Lemma~7.2]{bravo-gillespie-hovey}, we may continue to find an increasing chain of (degreewise split) inclusions $0 \neq Q^0 \subseteq Q^1
\subseteq Q^2 \subseteq \cdots $ corresponding to a nested union of
subsets $L^0_n \subseteq L^1_n \subseteq L^2_n \subseteq \cdots$ for
each $n$. Assuming this process doesn't terminate at a finite step, we set $Q^{\omega} =
\cup_{\alpha < \omega} Q^{\alpha}$ and note that $Q^{\omega}_n =
\oplus_{i \in L^{\omega}_n} M_i$ where $L^{\omega}_n =
\cup_{\alpha < \omega} L^{\alpha}_n$. So still, $Q^{\omega}$ and
$F/Q^{\omega}$ are complexes of fp-free modules and are acyclic since they are direct limits of acyclic complexes. Therefore we
can continue this process with $F/Q^{\omega}$ to obtain $Q^{\omega
+1}$ with all the same properties we desire. Using this process we can
obtain an ordinal $\lambda$ and a continuous union $F = \cup_{\alpha <
\lambda} Q^{\alpha}$ with each $Q_{\alpha}, Q_{\alpha + 1}/Q_{\alpha}$
acyclic complexes of fp-frees and having $|Q_0|,
|Q_{\alpha +1}/Q_{\alpha}| < \kappa$.
\end{proof}

\begin{theorem}\label{theorem-acyclic complexes of pure-projectives}
$(\exclass{PP}, \class{V})$ is a complete cotorsion pair, cogenerated by a set $\class{S}$ in the exact category $\ch_{pur}$. 
\end{theorem}

 \begin{proof}
For a chain complex $X$,  we define its cardinality to be $|X| :=|\coprod_{n\in\Z} X_n|$.
Let  $\kappa > \text{max}\{\, |R| \,  , \, \omega \,\}$ be a regular cardinal as in Proposition~\ref{prop-filtrations for complexes of pure-projectives}. Up to isomorphism, we can find a set $\class{S}$ (as opposed to a proper class) of complexes $X \in \exclass{PP}$ with cardinality $|X| \leq \kappa$. By Proposition~\ref{prop-cogeneration-by-set}, such a set $\class{S}$ cogenerates a functorially complete cotorsion pair $(\leftperp{(\rightperp{\class{S}})}, \rightperp{\class{S}})$ in the exact category $\ch_{pur}$. Moreover, the class $\leftperp{(\rightperp{\class{S}})}$ consists precisely of direct summands of transfinite degreewise pure extensions of complexes in $\class{S}$. We will show that $(\leftperp{(\rightperp{\class{S}})}, \rightperp{\class{S}}) = (\exclass{PP}, \class{V})$.

Of course, it is enough to show $\leftperp{(\rightperp{\class{S}})} = \exclass{PP}$. The containment $\leftperp{(\rightperp{\class{S}})} \subseteq \exclass{PP}$ follows from the containment $\class{S} \subseteq \exclass{PP}$ and the fact that $\exclass{PP}$ is closed under direct summands and transfinite degreewise pure extensions. This is because, as previously noted in the proof of Theorem~\ref{thm-bounded above complexes of finitely presented cogenerate}, the class $\class{PP}$ of pure-projective modules is closed under direct summands and transfinite pure extensions. Acyclic complexes are also closed under these operations, so $\exclass{PP}$ is closed under direct summands and transfinite extensions in the exact category $\ch_{pur}$.

On the other hand, it follows from Lemma~\ref{lemma-complexes of pure-projectives are retracts of complexes of fp-frees} and  Proposition~\ref{prop-filtrations for complexes of pure-projectives} that every $X \in \exclass{PP}$ is a direct summand of a transfinite degreewise pure extension of complexes in $\class{S}$. Therefore, $\exclass{PP} \subseteq \leftperp{(\rightperp{\class{S}})}$. 

This proves $(\exclass{PP}, \class{V})$ is a (functorially) complete cotorsion pair in the exact category $\ch_{pur}$, and it is cogenerated by the set $\class{S}$.
\end{proof}

Moreover, just like  $(\dwclass{PP}, \class{A}_{pur})$, we have that $(\exclass{PP}, \class{V})$ is a projective cotorsion pair in $\ch_{pur}$. Again, it just means that the triple $(\exclass{PP}, \class{V}, All)$ is an abelian model structure on the exact category $\ch_{pur}$. This all follows from Theorem~\ref{theorem-acyclic complexes of pure-projectives} by taking $\class{S}$ to be the set in the proof of Theorem~\ref{them-proj-cot-pair}. We summarize its main properties in the following statement.

\begin{corollary}\label{cor-exPP-model}
The triple $(\exclass{PP}, \class{V}, All)$ is a abelian model structure on the exact category $\ch_{pur}$ and satisfies each of the following properties:
\begin{enumerate}
\item A chain map $f : X \xrightarrow{} Y$  is a weak equivalence if and only if its cone $C(f) \in \class{V}$, and $\class{V}$ is a thick subcategory of the homotopy category $K(R)$.
\item Its homotopy category is equivalent to the Verdier quotient $K(R)/\class{V}$ and the model structure provides a canonical equivalence $K(R)/\class{V} \cong  K_{ac}(\class{PP})$, where $K_{ac}(\class{PP})$ is the full subcategory of $K(R)$ consisting of all complexes (homotopy equivalent to one) in $\exclass{PP}$. 
\item The model structure is cofibrantly generated and so $K(R)/\class{V}$ is a well-generated triangulated category. 
\end{enumerate}
\end{corollary}

In fact, we will show in Theorem~\ref{theorem-recollement-pp} that the triangulated category $K(R)/\class{V} \cong  K_{ac}(\class{PP})$ is compactly generated.

\section{DG-pure-projectives, recollement, and compact generation}\label{sec-dg-pp}

The cofibrant objects in the standard model structure on $\ch$ for the usual derived category, $\class{D}(R)$, are the DG-projective complexes.
Recall that a chain complex  $P$ is called \emph{DG-projective} if each $P_n$ is a projective $R$-module and if any chain map $P \xrightarrow{} E$, to any exact complex $E$, is null homotopic. In the same way we make the following definition.

\begin{definition}
A chain complex $P$ is called \emph{DG-pure-projective} if each $P_n$ is a pure-projective $R$-module and if any chain map $P \xrightarrow{} E$, to any exact complex $E$, is null homotopic. 
\end{definition}

Clearly any DG-projective complex is  DG-pure-projective. We note that the 
DG-projective complexes are precisely the K-projective complexes (in the sense of Spaltenstein~\cite{spaltenstein}) having projective components. In the same way we have the following.   

\begin{lemma}\label{lemma-DG-pp}
A chain complex $P$ is DG-pure-projective if and only if it is K-projective and each component $P_n$ is a pure-projective module.
\end{lemma}
 
 \begin{proof}
 The definition of DG-pure-projective is equivalent to the statement that each $P_n$ is a pure-projective and $\homcomplex(P,E)$ is acyclic for all exact complexes $E$. This is equivalent to saying each $P_n$ is pure-projective and the morphism set $K(R)(P,E)=0$ for all exact complexes $E$. This means $P$ is K-projective, by definition.  
 \end{proof}
 
Recall that $\{S^n(R)\}$ cogenerates the usual projective model structure on $\ch$, where the cofibrant objects are the DG-projectives. (See~\cite[Example~3.3]{hovey}.) The next theorem is essentially constructing the same model structure, except it is abelian with respect to the exact category $\ch_{pur}$.

 \begin{proposition}\label{prop-DG-pure-projective-cot-pair}
  Let $\class{E}$ denote the class of all exact (acyclic) chain complexes, and let $\class{C}$ denote the class of all DG-pure-projective complexes. Then $(\class{C}, \class{E})$ is a complete cotorsion pair in the exact category $\ch_{pur}$. 
  
  This cotorsion pair is cogenerated by the set $\class{S} = \{S^n(R)\}$ in $\ch_{pur}$. Consequently, each DG-pure-projective complex is a direct summand of a transfinite degreewise pure extension of complexes in  $\class{S}\cup \{D^n(M_i)\}$, where $\{M_i\}$ is some representative set of all the isomorphism classes of finitely presented $R$-modules.
\end{proposition}

 \begin{proof}
By Proposition~\ref{prop-cogeneration-by-set}, the set $\class{S}  = \{S^n(R)\}$ cogenerates a functorially complete cotorsion pair $(\leftperp{(\rightperp{\class{S}})}, \rightperp{\class{S}})$ in the exact category $\ch_{pur}$. Moreover, the class $\leftperp{(\rightperp{\class{S}})}$ consists precisely of direct summands of transfinite degreewise pure extensions of complexes in  $\class{S}\cup \{D^n(M_i)\}$, where $\{M_i\}$ is some representative set of all the isomorphism classes of finitely presented $R$-modules. 
We will show that $(\leftperp{(\rightperp{\class{S}})}, \rightperp{\class{S}}) = (\class{C}, \class{E})$. 
 
 First note that $X \in \rightperp{\class{S}}$ if and only if for all integers $n$ we have $$0 = \Ext^1_{pur}(S^n(R),X) = \Ext^1_{dw}(S^n(R),X) \cong  H_{n-1}\Hom_R(R,X) \cong H_{n-1}X$$ and so $\rightperp{\class{S}} = \class{E}$.
 
 It is left to see why $\leftperp{\class{E}} = \class{C}$. Note first that $X$ is a complex of pure-projectives if and only if for any module $M$ we have 
  $$0 = \Ext^1_{\class{P}ur}(X_n,M) \cong \Ext^1_{pur}(X, D^{n+1}(M))$$ where the isomorphism holds 
by~\cite[Lemma~4.2(2)]{gillespie-G-derived}.  Since any disk $D^{n+1}(M)$ is in the class $\class{E}$, and since $\class{E}$ is closed under suspensions we conclude that $\leftperp{\class{E}}$ consists precisely of the complexes of pure-projectives satisfying $\Ext^1_{pur}(X,\Sigma^n E) = \Ext^1_{dw}(X,\Sigma^n E) = 0$ for each $n$ and $E \in \class{E}$. But by a standard fact, for example see~\cite[Lemma~4.3]{gillespie-G-derived}, it is equivalent to state that $X$ is a complex of pure-projectives for which $H_n[\homcomplex(X,E)] = 0$ for all exact complexes $E$. This proves $\leftperp{\class{E}} = \class{C}$.
 \end{proof}

Again, we have that $(\class{C}, \class{E})$ is a projective cotorsion pair in $\ch_{pur}$.  
This all follows from Theorem~\ref{them-proj-cot-pair} by taking $\class{S}$ to be the set in the statement of Proposition~\ref{prop-DG-pure-projective-cot-pair}. We summarize its main properties in the following statement.

 \begin{corollary}\label{cor-DG-pure-projective-model}
 Let $\class{E}$ denote the class of all exact chain complexes, and let $\class{C}$ denote the class of all DG-pure-projective complexes. 
The triple $(\class{C}, \class{E}, All)$ is an abelian model structure on the exact category $\ch_{pur}$ and satisfies each of the following properties:
\begin{enumerate}
\item Its homotopy category recovers $\class{D}(R) := K(R)/\class{E}$, the usual derived category of $R$.
\item The model structure is finitely generated in the sense of~\cite[Section~7.4]{hovey-model-categories}. 
\item In the case that $R$ is commutative, the model structure is monoidal with respect to the usual tensor product of chain complexes. 
\end{enumerate}
\end{corollary}

\begin{proof}
Let $R$ be commutative and let us just comment on why the monoidal criteria of Theorem~\ref{them-proj-cot-pair} is satisfied. 
It is only required to show that $X \tensor_R Y \in \class{C}$ whenever $X,Y \in \class{C}$.  Certainly, $X \tensor_R Y$ is a complex of pure-projectives as shown in the proof of Corollary~\ref{cor-pure-proj-model}. So by Lemma~\ref{lemma-DG-pp} it only remains to show that $X \tensor_R Y$ is a K-projective complex. But it is standard that K-projectives are closed under tensor products. For examples, it follows from the enriched adjoint associativity isomorphism: $\homcomplex(X \tensor_R Y, E) \cong \homcomplex(X ,\homcomplex(Y,E))$.
\end{proof}
 
Returning to the notation of Section~\ref{sec-acyclic-pp}, recall that we have set $\class{V} = \rightperp{\exclass{PP}}$, where $\exclass{PP}$ denotes the class of all acyclic complexes of pure-projective $R$-modules. Then $\class{V}$ is the thick subcategory of $K(R)$ consisting of the K-absolutely pure complexes to appear in~\cite{emmanouil-kaperonis-K-flatness-pure}. The Verdier quotient satisfies $K(R)/\class{V} \cong K_{ac}(\class{PP})$, where $K_{ac}(\class{PP})$ is the homotopy category of all acyclic complexes of pure-projectives. 

\begin{theorem}\label{theorem-recollement-pp}
We have a recollement of triangulated categories
\[
\xy
(-28,0)*+{K(R)/\class{V}};
(0,0)*+{\class{D}_{pur}(R)};
(25,0)*+{\class{D}(R)};
{(-19,0) \ar (-10,0)};
{(-10,0) \ar@<0.5em> (-19,0)};
{(-10,0) \ar@<-0.5em> (-19,0)};
{(10,0) \ar (19,0)};
{(19,0) \ar@<0.5em> (10,0)};
{(19,0) \ar@<-0.5em> (10,0)};
\endxy
\] which when restricting the first two categories to cofibrant objects becomes 
\[
\xy
(-28,0)*+{K_{ac}(\class{PP})};
(0,0)*+{K(\class{PP})};
(25,0)*+{\class{D}(R)};
{(-19,0) \ar (-10,0)};
{(-10,0) \ar@<0.5em> (-19,0)};
{(-10,0) \ar@<-0.5em> (-19,0)};
{(10,0) \ar (19,0)};
{(19,0) \ar@<0.5em> (10,0)};
{(19,0) \ar@<-0.5em> (10,0)};
\endxy
.\]  
In particular, the Verdier quotient $K(R)/\class{V} \cong  K_{ac}(\class{PP})$ is compactly generated. 
\end{theorem}

\begin{proof}
The recollement follows from~\cite[Theorem~4.7]{gillespie-recollement} or~\cite[Theorem~3.5]{gillespie-recoll2}:  For we have three projective cotorsion pairs $$\mathfrak{M}_1 = (\dwclass{PP}, \class{A}_{pur}) \ , \ \ \ \ \mathfrak{M}_{2} = (\exclass{PP}, \class{V}) \ , \ \ \ \ \mathfrak{M}_{3} = (\class{C}, \class{E}),$$
coming respectively from Corollary~\ref{cor-pure-proj-model}, Corollary~\ref{cor-exPP-model}, and Corollary~\ref{cor-DG-pure-projective-model}.
 They satisfy the hypotheses of~\cite[Theorem~3.5]{gillespie-recoll2} so we have a recollement. 

Let us show that $K(R)/\class{V} \cong  K_{ac}(\class{PP})$ is compactly generated. 
As we have already pointed out, the pure derived category $\class{D}_{pur}(R)$, and the usual derived category $\class{D}(R)$, are compactly generated. It follows that $\textnormal{Ho}(\mathfrak{M}_2) \cong K(R)/\class{V} \cong K_{ac}(\class{PP}) $ is compactly generated too. Indeed the inclusion $I : K_{ac}(\class{PP}) \xrightarrow{} K(\class{PP})$ has both a left and a right adjoint, $I$ clearly preserves coproducts, and $K(\class{PP})$ is compactly generated (being equivalent to $\class{D}_{pur}(R)$). Now it is an exercise to check that the left adjoint of $I$ carries a set of compact weak generators for $K(\class{PP})$ to a set of compact weak generators for $K_{ac}(\class{PP})$.
\end{proof} 

\begin{remark}
We can find an explicit set of compact weak generators for $K_{ac}(\class{PP})$ as follows. Let $\class{T} = \{S^n(M_i)\}$ be the set of all spheres on any set $\{M_i\}$ of isomorphic representatives for the class of all finitely presented $R$-modules. By Corollary~\ref{cor-pure-proj-model}, $\class{T}$ is a set of compact week generators for $K(\class{PP})$. Then~\cite[Theorem~3.5]{gillespie-recoll2} shows that the left adjoint of the inclusion $I : K_{ac}(\class{PP}) \xrightarrow{} K(\class{PP})$ is obtained by taking a special $\class{E}$-preenvelope using enough injectives of the cotorsion pair $\mathfrak{M}_{3} = (\class{C}, \class{E})$. In more detail, the image of any $S^n(M_i) \in \{S^n(M_i)\}$ under the left adjoint $\lambda$, is obtained by way of a short exact sequence  in $\ch_{pur}$
$$ 0\xrightarrow{} S^n(M_i) \xrightarrow{} E \xrightarrow{} C \xrightarrow{}0$$ where $E$ is exact and $C$ is DG-pure-projective. Note then that $E$ is a (degreewise pure)  extension of two complexes in $\dwclass{PP}$, and hence $E \in \class{E} \cap \dwclass{PP} = \exclass{PP}$. Denoting such a complex $E$ by $\lambda(S^n(M_i))$, then we have that $\{\lambda(S^n(M_i))\}$ is a set of compact weak generators for $K_{ac}(\class{PP})$.
\end{remark}

\subsection{The K-absolutely pure model structure for the derived category} 
We end by noting that the class $\class{V}$ of K-absolutely pure compexes are the fibrant objects of an abelian model structure for $\class{D}(R)$, the usual derived catgory of $R$. The class $\dwclass{PP}$ of all complexes of pure-projectives are the cofibrant objects.

\begin{theorem}\label{theorem-derived-cat-K-pur}
 Let $R$ be any ring and $\class{D}(R)$ denote its derived category. 
 There is a cofibrantly generated abelian model structure on the exact category $\ch_{pur}$ represented by the Hovey triple $$\mathfrak{M}^{K\text{-abs}}_{\textnormal{der}} = (\dwclass{PP},\class{E},\class{V}).$$
We have triangulated category equivalences
$$\class{D}(R) \cong \textnormal{Ho}(\mathfrak{M}^{K\text{-abs}}_{\textnormal{der}}) \cong K(\class{PP}) \cap \class{V},$$
that is, $\class{D}(R)$ is equivalent to the chain homotopy category of all K-absolutely pure complexes with pure-projective components.
 \end{theorem}
 
 \begin{proof}
 Putting together what we have already shown in Theorems~\ref{thm-bounded above complexes of finitely presented cogenerate} and~\ref{theorem-acyclic complexes of pure-projectives}, the proof is routine. We have $\dwclass{PP}\cap\class{E} = \exclass{PP}$ which is left orthogonal to $\class{V}$. We only need need to confirm that $\class{E}\cap\class{V} = \class{A}_{pur}$ is the class of all pure acyclic complexes. Clearly $\class{E}\cap\class{V} \supseteq \class{A}_{pur}$. To show $\class{E}\cap\class{V} \subseteq  \class{A}_{pur}$, let $X \in \class{E}\cap\class{V}$, and use completeness of the cotorsion pair $(\dwclass{PP},\class{A}_{pur})$ to write a (degreewise pure) short exact sequence 
 $$0 \xrightarrow{} X \xrightarrow{} A \xrightarrow{} P \xrightarrow{} 0$$ where $A \in \class{A}_{pur}$ and $P \in \dwclass{PP}$. 
It is easy to argue that the short exact sequence must split, forcing $X$ to be a direct summand of $A$, and so is itself pure acyclic. 
 \end{proof}
 
 \section*{Acknowledgements}
 Upon completing this work, the author learned of the independent work of Ioannis Emmanouil and his student Ilias Kaperonis~\cite{emmanouil-kaperonis-K-flatness-pure}. The author thanks Emmanouil and Kaperonis for sharing their preliminary work on K-absolutely pure complexes which will appear in a forthcoming preprint~\cite{emmanouil-kaperonis-K-flatness-pure}. Their article is a must read and complements the present article.


\end{document}